\UseRawInputEncoding  
\documentclass[12pt]{amsart}
\usepackage{nccmath}

\usepackage{amsmath,amsthm,amssymb,bm}
\usepackage{times}
\usepackage{enumerate}
\usepackage{accents}

\usepackage{xcolor}

\newtheorem{theorem}{Theorem}[section]

\newtheorem{lemma}[theorem]{Lemma}
\newtheorem{proposition}[theorem]{Proposition}

\newcommand{\be}{\begin{equation}}
\newcommand{\ee}{\end{equation}}
\newcommand{\ju}[2]{\begin{array}{#1}#2\end{array}}

\newcommand{\ubar}[1]{\underaccent{\bar}{#1}}

\newcommand{\lt}{\left}
\newcommand{\rt}{\right}

\newcommand{\goto}{\rightarrow}
\newcommand{\R}{\mathbb{R}}

\newcommand{\e}{\epsilon}
\newcommand{\s}{\sigma}
\newcommand{\p}{\partial}
\newcommand{\al}{\alpha}

\newcommand{\ba}{\mathfrak b}
\newcommand{\lp}{\ubar{\phi}}
\newcommand{\lh}{\ubar{h}}
\newcommand{\la}{\lambda}
\newcommand{\tvarphi}{\tilde{\varphi}}
\newcommand{\La}{\Lambda}
\newcommand{\lu}{\ubar{u}}
\newcommand{\tx}{\tilde{x}}
\newcommand{\ta}{\tilde\alpha}
\newcommand{\tb}{\tilde\beta}
\newcommand{\tn}{\tilde n}
\newcommand{\tu}{\tilde u}

\newcommand{\oab}{\omega_{\alpha, \beta}}
\newcommand{\om}{\omega}

\newcommand{\RNum}[1]{\uppercase\expandafter{\romannumeral #1\relax}}

\theoremstyle{definition}
\newtheorem{defin}[theorem]{Definition}
\newtheorem{remark}[theorem]{Remark}

\numberwithin{equation}{section}

\begin{document}
\setlength{\baselineskip}{1.2\baselineskip}

\title[]
{Exterior Dirichlet problem for Hessian equations on a non-convex ring}

\author{Yanyan Li}
\address{Department of Mathematics, Rutgers University, Piscataway, NJ 08854}
\thanks{Y.Y. Li is partially supported by NSF grant DMS-2247410}
\email{yyli@math.rutgers.edu}

\author{Ling Xiao}
\address{Department of Mathematics, University of Connecticut, Storrs, CT 06269}
\email{ling.2.xiao@uconn.edu}
\maketitle
\begin{center}
Dedicated to Gang Tian on his 65th birthday with friendship and respect.
\end{center}

\begin{abstract}
In this paper, we prove the existence of a solution for the exterior Dirichlet problem for Hessian equations on a non-convex ring. Moreover, the solution we obtained
is smooth.
This extends the result of [Bao-Li-Li, ``On the exterior Dirichlet problem for Hessian equations''
Trans. Amer. Math. Soc.366(2014)].

\end{abstract}

\maketitle

\section{Introduction}
\label{sec1}
In this paper, $D^2u$ denotes the Hessian of $u,$ and the $k$-th elementary symmetric function $\s_k(A)$ of a symmetric matrix $A$ is defined by
\[\s_k(A)=\s_k(\la(A))=\sum\limits_{1\leq i_1<\cdots<i_k\leq n}\la_{i_1}\cdots\la_{i_k},\]
where $\la(A)=(\la_1, \cdots, \la_n)$ are the eigenvalues of $A$. Let $\Gamma_k$ be the Garding's cone
\[\Gamma_k=\{\la\in\mathbb R^{n}: \s_m(\la)>0,\,\,m=1, \cdots, k\}.\] It is well known that
$\Gamma_k$ is a convex symmetric cone with vertex at the origin.

In this paper, we consider the solvability of the following exterior Dirichlet problem for Hessian equations
\be\label{1.1}
\left\{\begin{aligned}
\s_k(\lambda(D^2u))&=1\,\,&\mbox{in $\mathbb R^n\setminus\bar D,$}\\
u&=\varphi\,\,&\mbox{on $\Gamma:=\partial D$.}
\end{aligned}
\right.
\ee

Before stating our main theorems, we will need the following definitions.
\begin{defin}
\label{def1.1}
A $C^2$ regular hypersurface $\mathcal{M}\subset\R^{n}$ is called \textbf{strictly k-convex} ($k$-convex) if its principal curvature vector
$\kappa(X)\in\Gamma_k$ ($\kappa(X)\in\bar\Gamma_k$) for all $X\in\mathcal{M}.$ We say that a domain $D$ is strictly $k$-convex ($k$-convex) if $\partial D$ is strictly $k$-convex ($k$-convex).
\end{defin}

\begin{defin}
\label{def1.2}
For any open set $U\subset\mathbb R^n,$ a function $v\in C^{2}(U)$ is called a  \textbf{strictly $k$-convex} ($k$-convex) function in $U$ if $\s_j(D^2 v(x))>0$ ($\s_j(D^2 v(x))\geq 0$) for all $j=1, \cdots, k$ and $x\in U.$
\end{defin}

\begin{defin}
\label{def1.3}
A function $u\in C^{0}(\mathbb R^n\setminus\bar D)$ is said to be a viscosity subsolution of \eqref{1.1} in $\mathbb R^n\setminus\bar D$, if for any $k$-convex function
$\psi\in C^2(\mathbb R^n\setminus\bar D)$ and $\bar x\in \mathbb R^n\setminus\bar D$ satisfying
\[\psi(\bar x)=u(\bar x)\,\,\mbox{and $\psi\geq u$ on $\mathbb R^n\setminus\bar D,$}\]
we have
\[\s_k(\la(D^2\psi(\bar x)))\geq 1.\]
Similarly, $u\in C^0(\mathbb R^n\setminus\bar D)$ is a viscosity supersolution of \eqref{1.1} in $\mathbb R^n\setminus\bar D,$ if for any $k$-convex function
$\psi\in C^2(\mathbb R^n\setminus\bar D)$ and $\bar x\in \mathbb R^n\setminus\bar D$ satisfying
\[\psi(\bar x)=u(\bar x)\,\,\mbox{and $\psi\leq u$ on $\mathbb R^n\setminus\bar D,$}\]
we have
\[\s_k(\la(D^2\psi(\bar x)))\leq 1.\]
We say $u$ is a viscosity solution of \eqref{1.1} if $u$ is both a viscosity subsolution and a viscosity supersolution of \eqref{1.1}.
\end{defin}

Following \cite{BLL14} we denote
\[\mathcal A_k=\{A: \mbox{$A$ is a real $n\times n$ symmetric positive definite matrix with $\s_k(\lambda(A))=1$}\}.\]
Our main results are

\begin{theorem}
\label{thm1}
Let $D$ be a bounded, smooth, star-shaped, strictly $(k-1)$-convex domain in $\mathbb R^n,$ $n\geq 3,$ and let
$\varphi\in C^{\infty}(\p D).$ Then for any given $b\in\mathbb R^n,$ $A\in\mathcal A_k$ with $2\leq k\leq n,$ 
there exists some constant $c_*=c_*(n, k, b, A, D, |\varphi|_{C^2(\p D)}),$ such that for every
$c>c_*,$ there exists a unique strictly $k$-convex solution $u\in C^{\infty}(\mathbb R^n\setminus D)$ of \eqref{1.1} satisfying
\be\label{1.2}
\limsup\limits_{|x|\goto\infty}|x|^{n-2}\lt[u-\lt(\frac{1}{2}x^TAx+b\cdot x+c\rt)\rt]<\infty.
\ee
\end{theorem}

\begin{theorem}
\label{thm2}
Let $u\in C^{\infty}(\mathbb R^n\setminus D)$ be the solution of \eqref{1.1} obtained in Theorem \ref{thm1} and denote
\[E(x):=u-\lt(\frac{1}{2}x^TAx+b\cdot x+c\rt).\] Then $E(x)$ satisfies
\be\label{1.3}
\limsup\limits_{|x|\goto\infty}|x|^{n-2+m}\lt|D^mE(x)\rt|<\infty
\ee
for any integer $m\geq 1.$
\end{theorem}

\begin{remark} For $k=n,$ our result extends Theorem 1.5 in \cite{CL03} by showing that the solution $u$ belongs to $C^{\infty}(\mathbb R^n\setminus D),$
whereas \cite{CL03} established $u\in C^{\infty}(\mathbb R^n\setminus\bar D).$
\end{remark}

By the discussions on page 4 of \cite{BLL14}, we know that we can always assume $A\in\mathcal A_k$ to be diagonal.
Now, let $s=\frac{1}{2}\sum\limits_{i=1}^na_ix_i^2,$ then for $A=\text{diag}(a_1, \cdots, a_n)\in\mathcal A_k$ we denote
\[E_\lambda:=\{x\in\mathbb R^n: s\equiv\frac{1}{2}\sum\limits_{i=1}^na_ix_i^2<\lambda, \lambda>0\}.\]
Moreover, in this paper, without loss of generality, we will always assume $b=\vec{0},$ $\bar D\subset E_1,$ and $D$ is star-shaped with respect to the origin. When $k=n,$
D is a strictly convex domain, which is automatically star-shaped with respect to any interior point.

The Dirichlet problem on exterior domains is closely related to the asymptotic
behavior of solutions defined on entire $\mathbb R^n.$ In \cite{CL03}, Caffarelli and the first named author proved that if $u$ is a convex viscosity solution of
\be\label{ad1.1}\det(D^2u)=1\ee outside a bounded subset of $\mathbb R^n,$ then
\[\limsup\limits_{|x|\goto\infty}|x|^{n-2}\lt[u-\lt(\frac{1}{2}x^TAx+b\cdot x+c\rt)\rt]<\infty,\,\,A\in\mathcal A_n.\]
Moreover, with such prescribed
asymptotic behavior near infinity, they also established an existence and uniqueness theorem for solutions
of \eqref{ad1.1}. Since then, the solvability of the exterior Dirichlet problem for fully nonlinear equations has been studied intensively, see for example \cite{DBW23, JLL21, LL18, Li19} and references therein. We remark that all these earlier works require that the domain $D$ is strictly convex. Moreover, the solutions found in these earlier works are viscosity solutions (except for the Monge--Amp\`{e}re case). While in this paper, the solutions we obtain are smooth and the domain $D$ is only required to be strictly $(k-1)$-convex.

The organization of the paper is as follows. In Section \ref{secsub}, we construct a viscosity subsolution for \eqref{1.1}.
In Section \ref{secthm1}, by solving \eqref{1.1} on bounded domains we prove the existence part of Theorem \ref{thm1}. In Section \ref{secthm2},
we study the asymptotic behavior of the solution near infinity. More precisely, we derive inequality \eqref{1.2} and prove Theorem \ref{thm2}.

\bigskip
\section{Construction of the subsolution of \eqref{1.1} }
\label{secsub}
In this section, we will combine ideas in  \cite{BLL14}, \cite{CL03}, and \cite{Xiao22} to construct subsolutions for \eqref{1.1}. More precisely,
we will first apply the techniques developed in Section 3 of \cite{Xiao22} to construct subsolutions of \eqref{1.1} in the region $E_1\setminus D.$ Then we will use the generalized symmetric subsolutions, which were studied in Section 2 of \cite{BLL14}, to construct subsolutions of \eqref{1.1} in the region $\mathbb R^n\setminus E_1.$ Finally, we use the idea from \cite{CL03} to glue the two subsolutions together and obtain a subsolution of \eqref{1.1}.

We want to note that, in this section, we will only explicitly construct the subsolution of
\eqref{1.1} for $2\leq k<n.$ When $k=n,$ the subsolution can be constructed in the same way but much easier. In particular, when $k=n,$ in the region $E_1\setminus D,$ since the domain $D$ is strictly convex, we can use the distance function, more precisely, $\text{dist}(x, \partial D)+1,$ instead of $\mathfrak b$ defined in \eqref{eq-b}; in the region $\mathbb R^n\setminus E_1$ we can use the symmetric subsolution that was studied in \cite{CL03} instead of the generalized symmetric subsolution.

\subsection{Subsolutions in $E_1\setminus D$}
\label{subsection-2.1}
Recall that $\Gamma:=\partial D$ is star-shaped and strictly $(k-1)$-convex, we can parametrize $\Gamma$ as a graph of the radial function
$\rho(\theta): \mathbb S^{n-1}\goto\mathbb R,$ i.e.,
\[\Gamma=\{\rho(\theta)\theta: \theta\in\mathbb S^{n-1}\}.\]
Then the second fundamental form of $\Gamma$ is (see \cite{BLO} for example)
\[h_{ij}=\frac{\rho}{w}\lt(\delta_{ij}+2\frac{\rho_i\rho_j}{\rho^2}-\frac{\rho_{i, j}}{\rho}\rt),\]
where $w=\sqrt{1+\frac{|\nabla\rho|^2}{\rho^2}},$ $\rho_{i, j}=\nabla_{ij}\rho,$
and $\nabla$ denotes the Levi-Civita connection on $\mathbb S^{n-1}.$

We denote $\Phi=\log\rho,$ it is clear that the second fundamental form of $\Gamma$ can be expressed as follows.
\[h_{ij}=\frac{\rho}{w}\lt(\delta_{ij}+\Phi_i\Phi_j-\Phi_{i,j}\rt),\]
where $w=\sqrt{1+|\nabla\Phi|^2}.$ By a direct calculation, we obtain
\be\label{add1}
g_{ij}=\rho^2(\delta_{ij}+\Phi_i\Phi_j),\,\,
g^{ij}=\frac{1}{\rho^2}\lt(\delta_{ij}-\frac{\Phi_i\Phi_j}{w^2}\rt),\,\,\text{and}\,\,\gamma^{ij}=\frac{1}{\rho}\lt(\delta_{ij}-\frac{\Phi_i\Phi_j}{w(1+w)}\rt).
\ee
Here, $(g_{ij})$ is the metric on $\Gamma,$ $(g^{ij})$ is the inverse of $(g_{ij}),$ and $\gamma^{ij}$ is the square root of $g^{ij},$
i.e., $\sum\limits_k\gamma^{ik}\gamma^{kj}=g^{ij}.$
Let $a_{ij}=\gamma^{ik}h_{kl}\gamma^{lj},$ then the eigenvalues of $(a_{ij})_{1\leq i, j\leq n-1},$ denoted by $\kappa[a_{ij}]=(\kappa_1, \cdots, \kappa_{n-1})$ are the principal curvatures of $\Gamma.$

The following calculation can be found in Section 3 of \cite{Xiao22}, for readers convenience, we include it here.
\subsubsection{Hessian in spherical coordinates}
\label{hess}
Let $f: \R^n\goto \R$ be a scalar function, then $f$ can also be expressed as a function of $(\theta, r)\in\mathbb{S}^{n-1}\times\R.$
Note that the Euclidean metric is $g_E=r^2dz^2+dr^2,$ where $dz^2$ is the standard metric on $\mathbb{S}^{n-1}.$ In the following, we will denote the standard connection in $\mathbb R^n$ by $D$.
Now, we choose a local orthonormal frame $\{e_1, \cdots, e_{n-1}\}$ on the unit sphere $\mathbb{S}^{n-1}$. Let $\tau_{a}=\dfrac{e_{a}}{r}$, $1\leq a\leq n-1,$ which is the orthonormal frame on the sphere with radius $r,$ and we also let $\tau_r=\frac{\partial}{\partial r}$.
Then a direct calculation yields the Hessian of $f$ in spherical coordinates is
\be\label{hess1.1}D^2_{ab}f=D^2f(\tau_{a},\tau_{b})=\frac{1}{r^2}f_{ab}+\frac{1}{r}f_r\delta_{ab},\ee
\be\label{hess1.2}D^2_{a r}f=D^2f(\tau_{a},\tau_r)=\frac{1}{r}f_{a r}-\frac{1}{r^2}f_{a},\ee
and
\be\label{hess1.3}D^2_{rr}f=D^2f(\tau_r,\tau_r)=f_{rr}.\ee
Here $1\leq a, b\leq n-1,$ $f_{ab}=e_{b}e_{a}f, f_{a r}=\tau_re_{a} f,$ and $f_{rr}=\tau_r\tau_rf$.

\subsubsection{Construction of subsolutions in $E_1\setminus D$}
Now, we fix an arbitrary point $p\in \mathbb{S}^{n-1},$ let $\{e_1, \cdots, e_{n-1}\}$ be the normal coordinates at $p,$ then the Christoffel symbols vanish at
$p.$ This implies at this point we get $\nabla_{ij}\rho=e_ie_j\rho.$ In the following, for any function $f$ defined in a small neighborhood of $p,$ we denote $f_{ij}=e_ie_j f,$ then we have $\rho_{i, j}=\rho_{ij}$ at $p.$   Moreover, we may rotate the coordinates such that $|\nabla\rho(p)|=\rho_1$ and $\rho_{\al\beta}(p)=\rho_{\al\al}\delta_{\al\beta}$
for $2\leq\al,\beta\leq n-1.$ Then at the point $\hat{p}=\rho(p)p\in\Gamma,$ in view of \eqref{add1} we have
\[
\left\{
\begin{aligned}
\gamma^{11}&=\frac{1}{\rho}\lt(1-\frac{w^2-1}{w(1+w)}\rt)=\frac{1}{\rho w},\\
\gamma^{1\al}&=0,\,\,&2\leq\al\leq n-1,\\
\gamma^{\al\beta}&=\frac{1}{\rho}\delta_{\al\beta},\,\,&2\leq\al, \beta\leq n-1,
\end{aligned}
\right.
\]
and
\be\label{second-fundamental-form}
\left\{
\begin{aligned}
a_{11}&=\gamma^{1k}h_{kl}\gamma^{l1}=\gamma^{11}h_{11}\gamma^{11}=\frac{h_{11}}{\rho^2 w^2},\\
a_{1\al}&=\gamma^{1k}h_{kl}\gamma^{l\alpha}=\frac{h_{1\al}}{\rho^2 w},\,\,&2\leq\al\leq n-1,\\
a_{\al\beta}&=\gamma^{\al\al}h_{\al\beta}\gamma^{\beta\beta}=\frac{1}{\rho^2}h_{\al\beta},\,\,&2\leq\al, \beta\leq n-1.
\end{aligned}
\right.
\ee
Now, let us consider the function  \be\label{eq-b}\mathfrak b=\frac{r}{\rho(\theta)}.\ee
By a straight forward calculation we obtain that at the point $(p, r)\in(\mathbb S^{n-1}\times R)\setminus\{0\}$
the Hessian of $\ba$ is (for details see Subsection 3.2 of \cite{Xiao22})
\[
\begin{aligned}
\text{Hessian}(\mathfrak b)&=\left[\ju{ccccc}{\frac{w^3}{r}a_{11}&\frac{w^2}{r}a_{12}&\cdots&\frac{w^2}{r}a_{1n-1}&0\\
\frac{w^2}{r}a_{12}&\frac{w}{r}a_{22}&\cdots&0&0\\
\vdots&\vdots&\ddots&\vdots&\vdots\\
\frac{w^2}{r}a_{1n-1}&0&\cdots&\frac{w}{r}a_{n-1n-1}&0\\
0&0&\cdots&0&0}\right].
\end{aligned}
\]
We want to point out that $\kappa[a_{ij}]$ are the principal curvatures of $\Gamma$ at $\hat{p}=\rho(p)p.$

Next, we will consider the function $\phi=\phi(\mathfrak b),$ where $\phi$ is a function defined on $\R.$
We will compute the Hessian of $\phi$ at $(p, r).$ Denote $\phi'|_{(p, r)}=\frac{d\phi}{d\ba}|_{(p, r)}=M,$ $\phi''|_{(p, r)}=\frac{d^2\phi}{d\ba^2}|_{(p, r)}=B,$ we get at this point, for $1\leq i,j\leq n$,
$$D^2_{ij}\phi=MD^2_{ij}\ba+B(\tau_i\ba)(\tau_j\ba).$$
Here, $\tau_a=\frac{e_a}{r}$ for $1\leq a\leq n-1,$ $\tau_n:=\tau_r=\frac{\partial}{\partial r},$ $\{e_1, \cdots, e_{n-1}\}$ be the normal coordinates at $p$ chosen above, and $\{\tau_1, \cdots, \tau_n\}$ forms an orthonormal frame of $\mathbb R^n$ at $(p, r).$
We denote $\rho_a:=e_a\rho,$ then at $(p, r)$ we have
 \[
\begin{aligned}
\text{Hessian}(\phi)&=\left[\ju{ccccc}{\frac{Mw^3}{r}a_{11}+B\rho^{-4}\rho_1^2&\frac{Mw^2}{r}a_{12}&\cdots&\frac{Mw^2}{r}a_{1n-1}&-B\rho^{-3}\rho_1\\
\frac{Mw^2}{r}a_{12}&\frac{Mw}{r}a_{22}&\cdots&0&0\\
\vdots&\vdots&\ddots&\vdots\\
\frac{Mw^2}{r}a_{1n-1}&0&\cdots&\frac{Mw}{r}a_{n-1n-1}&0\\
-B\rho^{-3}\rho_1&0&\cdots&0&B\rho^{-2}}\right].
\end{aligned}
\]
Therefore, following the calculation on Subsection 3.2 of \cite{Xiao22} we get for any $2\leq m\leq k$
\be\label{sub1}
\begin{aligned}
\s_m(D^2\phi)&=\frac{M^mw^{m+2}}{r^m}[\s_m(a_{\hat i\hat j})-\s_m(a_{\alpha\beta})]\\
&+B\rho^{-2}\frac{M^{m-1}w^{m+1}}{r^{m-1}}\s_{m-1}(a_{\hat i\hat j})+\lt(\frac{Mw}{r}\rt)^m\s_m(a_{\alpha\beta}),
\end{aligned}
\ee
where $2\leq\alpha, \beta\leq n-1$ and $1\leq\hat i, \hat j\leq n-1.$
Moreover, we have
\be\label{add2}
\s_1(D^2\phi)=\frac{Mw^3}{r}a_{11}+\frac{Mw}{r}\s_1(a_{\alpha\beta})+B\rho^{-2}+B\rho^{-4}\rho_1^2.
\ee
This yields for any $1\leq m\leq k,$
\be\label{sub2}\s_m(D^2\phi)>-c_0\lt(\frac{M}{r}\rt)^m+c_1\frac{B}{\rho^2}\lt(\frac{M}{r}\rt)^{m-1}=\frac{M^{m-1}}{r^{m-1}}\lt(\frac{c_1 B}{\rho^2}-c_0\frac{M}{r}\rt),\ee
where $c_1=\min\limits_{\hat q\in\Gamma}\s_{m-1}(\kappa(\hat q))>0$ and $c_0=c_0(|\rho|_{C^2})>0$ are two positive constants that only depend on $\Gamma.$
We note that in this paper we use the convention $\s_0=1,$ and when $m=1$ we let $c_1=1.$

\begin{proposition}
\label{prosub1}
Let $D$ be a smooth, star-shaped, strictly $(k-1)$-convex domain in $\mathbb R^n$ for $n\geq 3,$ and let
$\varphi\in C^{\infty}(\p D).$ Then given any $A\in\mathcal A_k$ with $2\leq k\leq n-1,$
there exists $N_1=N_1(\varphi, \Gamma, A)>0$ such that when $N\geq N_1,$ $\lp=\ba^N-1+\varphi$ is strictly $k$-convex in $E_1\setminus \bar D$. Moreover, $\lp$ satisfies
$\s_k(\lambda(D^2\lp))>1\,\,\mbox{in $E_1\setminus \bar D$}$ and $\lp=\varphi$ on $\Gamma.$
\end{proposition}
\begin{proof}
Since $\Gamma$ is star-shaped, it is clear that $\varphi$ can be viewed as a function defined on $\mathbb S^{n-1}.$
We may extend the domain of definition of $\varphi$ to $\mathbb R^n\setminus\{0\}$ by setting $\varphi(r, \theta):=\varphi(\theta).$ Without causing any confusions, we will denote the extension of $\varphi$ by $\varphi.$

\textbf{Claim:} Let $\phi=\ba^N$ and denote $\lambda(D^2(\phi))=(\lambda_1, \cdots, \lambda_n)$ to be the eigenvalues of $\text{Hessian}(\phi),$ then $\lambda_a=O(M)$ for $a\leq n-1$ and $\lambda_n=B(\rho^{-4}|\nabla\rho|^2+\rho^{-2})+O(M).$ Here,
$M=N\ba^{N-1}$ and $B=N(N-1)\ba^{N-2}.$

\text{\em Proof of the claim}: Following the proof of Lemma 1.2 in \cite{CNS3}, let's consider the eigenvalues of the following matrix
 \[Q=\left[\ju{ccccc}{d_1+aX^2&a_2&\cdots&a_{n-1}&-aXY\\
a_2&d_2&\cdots&0&0\\
\vdots&\vdots&\ddots&\vdots\\
a_{n-1}&0&\cdots&d_{n-1}&0\\
-aXY&0&\cdots&0&aY^2}\right]\]
with $d_1, \cdots, d_{n-1},$ $X, Y,$ and $a_2, \cdots, a_{n-1}$ being fixed.
Denote
\[f_a(\lambda):=\det\left[\ju{ccccc}{\frac{d_1}{a}+X^2-\frac{\lambda}{a}&a_2&\cdots&a_{n-1}&-aXY\\
\frac{a_2}{a}&d_2-\lambda&\cdots&0&0\\
\vdots&\vdots&\ddots&\vdots\\
\frac{a_{n-1}}{a}&0&\cdots&d_{n-1}-\lambda&0\\
-XY&0&\cdots&0&aY^2-\lambda}\right],\]
the eigenvalues of $Q$ satisfies $f_a(\lambda)=0.$
A direct calculation yields
\[\begin{aligned}
f_a(\lambda)&=\left[\left(X^2+d_1/a-\lambda/a\right)(aY^2-\lambda)-aX^2Y^2\right]\prod\limits_{\alpha=2}^{n-1}(d_\alpha-\lambda)\\
&-\sum\limits_{\beta=2}^{n-1}\frac{a^2_{\beta}}{d_\beta-\lambda}\prod\limits_{\alpha=2}^{n-1}(d_\alpha-\lambda)\cdot(Y^2-\lambda/a).
\end{aligned}\]
For $a=\infty,$ $f_{\infty}(\lambda)$ is a polynomial of degree $n-1$ with coefficients only depending on $d_1, \cdots, d_{n-1},$ $X, Y,$ and $a_2, \cdots, a_{n-1}.$ Therefore, it has $n-1$ roots $\tilde d_1, \cdots, \tilde d_{n-1}$ that are depending on $d_1, \cdots, d_{n-1},$ $X, Y,$ and $a_2, \cdots, a_{n-1}.$ By continuity of the roods it follows that
$\lambda_a=\tilde d_a+o(1)$ for $a\leq n-1.$

To find the last eigenvalue set $\lambda=a\mu$. Then $\mu$ satisfies
\[\det\left[\ju{ccccc}{\frac{d_1}{a}+X^2-\mu &\frac{a_2}{a}&\cdots&\frac{a_{n-1}}{a}&-XY\\
\frac{a_2}{a}&\frac{d_2}{a}-\mu&\cdots&0&0\\
\vdots&\vdots&\ddots&\vdots\\
\frac{a_{n-1}}{a}&0&\cdots&\frac{d_{n-1}}{a}-\mu&0\\
-XY&0&\cdots&0&Y^2-\mu}\right]=0.\]
For $a=\infty,$ we see that $\mu=X^2+Y^2$ is a simple root. By the implicity function theorem it follows that for $a>0$ large there is a root
$\mu=X^2+Y^2+O(\frac{1}{a}),$ i.e., $\lambda_n=a\left(X^2+Y^2+O(\frac{1}{a})\right).$

Now, recall that $\ba$ is uniformly bounded in $E_1\setminus\bar D,$ we have
$B/M=(N-1)/\ba\goto \infty$ as $N\goto\infty.$ Therefore, we can apply the conclusion from above discussions to the matrix $\frac{1}{M}\text{Hessian}(\phi).$
In this case, correspondingly we have
\[d_1=\frac{w^3}{r}a_{11}, d_i=\frac{w}{r}a_{ii}, a_i=\frac{w^2}{r}a_{1i}\,\,\mbox{for $2\leq i\leq n-1$,} \]
and
\[a=\frac{B}{M}, X=\rho^{-2}\rho_1, Y=\rho^{-1}.\]
Note that if $\lambda$ is an eigenvalue of $\frac{1}{M}\text{Hessian}(\phi),$ then $M\lambda$ is an eigenvalue of $\text{Hessian}(\phi).$ Moreover, by our choice of coordinates we have $|\nabla\rho(p)|=\rho_1.$ The claim follows immediately.

By the Claim we can see that for all $1\leq m\leq k$
\[\sigma_m(\lambda(D^2\lp))\geq\sigma_m(\lambda(D^2\phi))-CM^{m-2}B\,\,\mbox{in $E_1\setminus\bar D,$}\]
where $C=C(\Gamma, A, \varphi)>0.$
In view of inequality \eqref{sub2} and equality \eqref{add2} we obtain there exists $N_1=N_1(\varphi, \Gamma, A)>0$ such that when $N\geq N_1,$ $\lp$ satisfies
\[\lambda(D^2\lp)\in\Gamma_k\,\,\mbox{and $\s_k(\lambda(D^2\lp))>1$ on $E_1\setminus\bar D$.}\]
This completes the proof of the proposition.

\end{proof}

\subsection{Subsolutions in $\mathbb R^n\setminus E_1$}
\label{subsection-2.2}
In the following, we will assume $N\geq N_1$ to be a fixed constant, we will also denote $\tvarphi:=\lp|_{\partial E_1}.$ Here, $N_1$ and $\lp$ are the same as the ones in Proposition
\ref{prosub1}.

First, recall that $s=\frac{1}{2}\sum a_ix_i^2,$ let
\[\oab(s)=\int_1^s(1+\alpha t^{-\beta})^{1/k}dt,\,\,\mbox{where $\alpha>0$ and $\beta>1$ to  be determined.}\]
In the following, when there is no confusion, we will drop the subscript $\alpha, \beta$ and write $\omega$ instead of $\oab.$
A direct calculation yields
\[\om'(s)=(1+\alpha s^{-\beta})^{1/k}\]
and
\[\om''(s)=\frac{1}{k}(1+\alpha s^{-\beta})^{1/k-1}(-\alpha\beta s^{-\beta-1})=-\frac{\alpha\beta \om'(s)}{k(s^{\beta+1}+\alpha s)}.\]
By page 5 of \cite{BLL14} we have
\[D_i\om(x)=\om'(s)a_ix_i=(1+\alpha s^{-\beta})^{1/k}a_ix_i\]
and
\be\label{add3}
\begin{aligned}
D_{ij}\om(x)&=\om'(s)a_i\delta_{ij}+\om''(s)(a_ix_i)(a_jx_j)\\
&=\om'(s)\lt[a_i\delta_{ij}-\frac{\alpha\beta}{ks(s^{\beta}+\alpha)}(a_ix_i)(a_jx_j)\rt].
\end{aligned}
\ee
In view of Proposition 1.2 of \cite{BLL14} we obtain
\be\label{sub5}\s_m(\lambda(D_{ij}\om))=[\om'(s)]^m\lt\{\s_m(a)-\frac{\alpha\beta}{ks(s^\beta+\alpha)}\sum\limits_{i=1}^n(a_ix_i)^2\s_{m-1}(a|i)\rt\},\ee
where $a=(a_1, \cdots, a_n)$ and $\s_{m-1}(a|i)=\s_{m-1}(a)|_{a_i=0}.$
Assume $a_{i_0}=\max\limits_{i}\{a_1, \cdots, a_n\},$ following \cite{BLL14}, let
\[
\begin{aligned}
h_m:&=\max\limits_{1\leq i\leq n}A_m^i(a)=\max\limits_{1\leq i\leq n}\s_{m-1}(a|i)a_i\\
&=\max\limits_{1\leq i\leq n}\{\s_m(a)-\s_m(a|i)\}=\s_m(a)-\s_m(a|i_0).
\end{aligned}
\]
Then we have
\be\label{sub6}\s_m(\lambda(D_{ij}\om))\geq[\om'(s)]^m\lt\{\s_m(a)-\frac{2\alpha\beta }{k(s^\beta+\alpha)}h_m\rt\}.\ee
In particular, when $m=k$ we get
\[\s_k(\lambda(D_{ij}\om))\geq\lt(1+\frac{\alpha}{s^\beta}\rt)\lt[1-\frac{2\alpha\beta h_k}{k(s^\beta+\alpha)}\rt].\]
According to equation (2.15) of \cite{BLL14}, we know that for $n\geq 3$ and $2\leq k\leq n-1,$ we have
\be\label{ad2.2*}\frac{k}{2}<\frac{k}{2h_k(a)}\leq\frac{n}{2}.\ee
In the following, we will set $\bm{\beta}=\mathbf{\frac{k}{2h_k}}-\bm{\eta},$ where $\eta>0$ is an arbitrary constant such that $\beta>k/2.$

In this case we have
\be\label{ad2.1}
\begin{aligned}
\s_k(\lambda(D_{ij}\om))&\geq\lt(1+\frac{\alpha}{s^\beta}\rt)\lt[1-\frac{2\alpha h_k}{k(s^\beta+\alpha)}\lt(\frac{k}{2h_k}-\eta\rt)\rt]\\
&=\lt(1+\frac{\alpha}{s^\beta}\rt)\lt[1-\frac{\alpha}{s^\beta+\alpha}+\frac{2\alpha h_k\eta}{k(s^\beta+\alpha)}\rt]\\
&=\frac{\alpha+s^\beta}{s^\beta}\lt[\frac{s^\beta}{s^\beta+\alpha}+\frac{2\alpha h_k\eta}{k(s^\beta+\alpha)}\rt]\\
&=1+\frac{2\alpha h_k\eta}{ks^\beta}.
\end{aligned}
\ee
Next, we will show that the $w_{\alpha,\beta}(s)$ constructed above  is a strictly $k$-convex function in $\mathbb R^n\setminus E_1.$
In view of \eqref{ad2.1}, we can see that we only need to show $\la(D^2_{ij}\omega)\in\Gamma_m$ in $\mathbb R^n\setminus E_1$ for $1\leq m<k.$
Note that
\[
\begin{aligned}
&\s_m(a)-\frac{\alpha\beta}{ks(s^\beta+\alpha)}\sum\limits_{i=1}^n\s_{m-1}(a|i)(a_ix_i)^2\\
&=\s_m(a)-\frac{\alpha}{ks(s^\beta+\alpha)}\sum\limits_{i=1}^n\s_{m-1}(a|i)(a_ix_i)^2\lt(\frac{k}{2h_k}-\eta\rt)\\
&=\s_m(a)-\frac{\alpha}{2h_ks(s^\beta+\alpha)}\sum\limits_{i=1}^n\s_{m-1}(a|i)(a_ix_i)^2+III\\
&\geq \s_m(a)-\frac{\alpha}{2s(s^\beta+\alpha)}\sum\limits_{i=1}^n\frac{\s_{m-1}(a|i)(a_ix_i^2)}{a_i\s_{k-1}(a|i)}+III,
\end{aligned}
\]
where $III:=\frac{\alpha\eta}{ks(s^\beta+\alpha)}\sum\limits_{i=1}^n\s_{m-1}(a|i)(a_ix_i)^2.$
By virtue of (2.22) of \cite{BLL14} we know
\[\frac{\s_{m-1}(a|i)}{\s_{k-1}(a|i)}\leq\s_m(a)\,\,\mbox{for $1\leq i\leq n$ and $1\leq m\leq k-1.$}\]
Therefore, we obtain
\be\label{sub3}
\begin{aligned}
&\s_m(a)-\frac{\alpha\beta}{ks(s^\beta+\alpha)}\sum\limits_{i=1}^n\s_{m-1}(a|i)(a_ix_i)^2\\
&\geq\s_m(a)-\frac{\alpha}{2s(s^\beta+\alpha)}\sum\limits_{i=1}^n\s_m(a)a_ix_i^2+III\\
&=\s_m(a)\lt(1-\frac{\alpha}{s^\beta+\alpha}\rt)+III.
\end{aligned}
\ee
Denote $\lh_m=\min\limits_{1\leq i\leq n}\s_{m-1}(a|i)a_i=\s_m(a)-\s_m(a|i_1),$
where $a_{i_1}=\min\{a_1, \cdots, a_n\}.$ Then
\[
\begin{aligned}
III&=\frac{\alpha\eta}{ks(s^\beta+\alpha)}\sum\limits_{i=1}^n\s_{m-1}(a|i)(a_ix_i)^2\\
&\geq\frac{2\alpha\eta\lh_m}{k(s^\beta+\alpha)}
\end{aligned}.
\]
Combining with \eqref{sub3} we have
\be\label{sub7}
\begin{aligned}
&\s_m(a)-\frac{\alpha\beta}{ks(s^\beta+\alpha)}\sum\limits_{i=1}^n\s_{m-1}(a|i)(a_ix_i)^2\\
&\geq\s_m(a)\frac{s^\beta}{s^{\beta}+\alpha}+\frac{2\alpha\eta\lh_m}{k(s^\beta+\alpha)}.
\end{aligned}
\ee
Thus, from \eqref{sub5} we get $\lambda(D_{ij}\om)\in\Gamma_k.$
We conclude
\begin{proposition}
\label{prosub2}
For $n\geq 3,$ $2\leq k\leq n-1,$ and $A\in\mathcal A_k,$ let $\omega_{\alpha, \beta}(x)=\int_1^s(1+\alpha t^{-\beta})^{1/k}dt$ for $s=\frac{1}{2}x^TAx.$
Then when $\alpha>0$, $\frac{k}{2}<\beta<\frac{k}{2h_k},$ $\oab$ is a smooth strictly $k$-convex subsolution of $\s_k(\lambda(D^2u))=1$ in $\mathbb R^n\setminus\{0\}.$ Moreover,
$\oab$ satisfies
\be\label{sub8}
\oab=\frac{1}{2}x^TAx+\mu(\alpha, \beta)+O(s^{1-\beta}),\,\,\mbox{as $s\goto\infty.$}
\ee
Here, $\mu(\alpha, \beta)=\int_1^{\infty}[(1+\alpha t^{-\beta})^{1/k}-1]dt-1.$
\end{proposition}
\begin{proof}
From \eqref{ad2.1} and \eqref{sub7} we know that $\oab$ is a smooth strictly $k$-convex subsolution of $\s_k(\lambda(D^2u))=1$ in $\mathbb R^n\setminus\{0\}.$
To prove this proposition, we only need to prove \eqref{sub8}. A straightforward calculation yields
\[
\begin{aligned}
\oab&=\int_1^s(1+\alpha t^{-\beta})^{1/k}dt\\
&=\int_1^s\lt[(1+\alpha t^{-\beta})^{1/k}-1\rt]dt+s-1\\
&=\int_1^\infty\lt[(1+\alpha t^{-\beta})^{1/k}-1\rt]dt-\int_s^\infty\lt[(1+\alpha t^{-\beta})^{1/k}-1\rt]dt+s-1\\
&=s+\mu(\alpha, \beta)+O(s^{1-\beta}).
\end{aligned}
\]
Here, $\beta=\frac{k}{2h_k}-\eta>\frac{k}{2}$ and
\[\mu(\alpha, \beta)=\int_1^\infty\lt[(1+\alpha t^{-\beta})^{1/k}-1\rt]-1<\infty.\]
\end{proof}
\begin{remark}
We want to point out that for any fixed $\frac{k}{2}<\beta<\frac{k}{2h_k},$ there exists $\alpha^*=\alpha^*(\beta, A, \varphi, D)>0$ such that when
$\alpha\geq\alpha^*,$ we have $\frac{1}{2}x^TAx+\mu(\alpha, \beta)\geq\varphi$ on $\Gamma.$
\end{remark}

Finally, we will use $\oab$ to construct a subsolution of $\s_k(\lambda(D^2u))=1$ in
$\mathbb R^n\setminus\bar E_1.$

Recall that we have set $\tvarphi=\lp|_{\partial E_1}.$ We may express $\tvarphi$ as a function of $\mathbb S^{n-1},$ i.e., $\tvarphi$ is independent of $r.$
Then set $\Psi=s^{-\Lambda}\tvarphi,$ where $\Lambda>0$ is a constant to be determined. Same as before, we choose a local orthonormal frame
$\{e_1, \cdots, e_{n-1}\}$ on the unit sphere $\mathbb S^{n-1}$ and let $\tau_r=\frac{\partial}{\partial r}.$
It is easy to see that for $1\leq a, b\leq n-1$
\[s_a=O(s),\,\,s_{ab}=O(s),\,\, s_{ar}=O(s^{1/2}),\]
\[s_r=O(s^{1/2}),\,\,\mbox{and $s_{rr}=O(1).$}\]
This implies
\[\begin{aligned}
\Psi_a&=-\Lambda s^{-\Lambda-1}\tvarphi s_a+s^{-\Lambda}\tvarphi_a=O(s^{-\Lambda}),\\
\Psi_{ab}&=\Lambda(\Lambda+1)s^{-\Lambda-2}\tvarphi s_as_b-\La s^{-\La-1}\tvarphi_bs_a\\
&-\La s^{-\La-1}\tvarphi s_{ab}-\La s^{-\La-1}\tvarphi_as_b+s^{-\La}\tvarphi_{ab}=O(s^{-\La}),\\
\Psi_{ar}&=-\La s^{-\La-1}\tvarphi s_{ar}+\La(\La+1)s^{-\La-2}\tvarphi s_as_r-\La s^{-\La-1}\tvarphi_as_r=O(s^{-\La-1/2})\\
\Psi_r&=-\La s^{-\La-1}s_r\tvarphi=O(s^{-\La-1/2})\\
\Psi_{rr}&=\La(\La+1)s^{-\La-2}(s_r)^2\tvarphi-\La s^{-\La-1}s_{rr}\tvarphi=O(s^{-\La-1})
\end{aligned}
\]
Consider
\[\lp^1=\int_1^s(1+\alpha t^{-\beta})^{1/k}dt+\Psi=\oab+\Psi,\]
then by \eqref{hess1.1}-\eqref{hess1.3} and \eqref{add3} we obtain
\[D_{ij}\lp^1=w'(s)\lt[a_i\delta_{ij}-\frac{\alpha\beta}{ks(s^\beta+\alpha)}(a_ix_i)(a_jx_j)\rt]+\frac{b_{ij}}{S^{\La+1}},\]
where $|b_{ij}|_{1\leq i, j\leq n}$ are uniformly bounded by some constant $C=C(|\tvarphi|_{C^2}, A, \La )>0.$
Now let's fix some $\La$ such that $\La+1\geq\beta>k/2,$ by virtue of \eqref{sub5}  we can see that
\be\label{add4}
\begin{aligned}
\s_m(\la(D^2\lp^1))&=\s_m(\la(D^2\oab+D^2\Psi))\\
&\geq\s_m(\la(D^2\oab))-\sum\limits_{i=1}^m\frac{C_i}{s^{(\La+1)i}}(\omega'(s))^{m-i}\\
&= \s_m(\la(D^2\oab))-\sum\limits_{i=1}^m\lt(1+\frac{\alpha}{s^\beta}\rt)^{\frac{m-i}{k}}\frac{C_i}{s^{(\La+1)i}},
\end{aligned}
\ee
where $C_i=C_i(A, \Lambda, |\tvarphi|_{C^2})>0.$
Denote $y:=\frac{\alpha}{s^{\beta}}$ and $\frac{\Lambda+1}{\beta}:=1+\gamma,$ when $m=k$ plugging \eqref{ad2.1} into \eqref{add4} gives
\begin{align*}
\s_k(\la(D^2\lp^1))&\geq 1+\frac{2h_k\eta}{k}y-\sum\limits_{i=1}^k(1+y)^{\frac{k-i}{k}}C_i\lt(\frac{y}{\alpha}\rt)^{(1+\gamma)i}\\
&\geq 1+y\lt[\frac{2h_k\eta}{k}-\sum\limits_{i=1}^k(1+y)^{\frac{k-i}{k}}C_i\lt(\frac{y}{\alpha}\rt)^{(1+\gamma)i-1}\frac{1}{\alpha}\rt]\\
&\geq 1+y\lt[\frac{2h_k\eta}{k}-\sum\limits_{i=1}^k(1+\alpha)^{\frac{k-i}{k}}C_i\frac{1}{\alpha}\rt].
\end{align*}
Here, we have used the fact that in $\mathbb R^n\setminus\bar E_1$ we have $0<y<\alpha.$ Therefore,
there exists $\alpha_0=\alpha_0(A, \Lambda, |\tilde\varphi|_{C^2},\beta, k)>0$ such that for any $\alpha>\alpha_0$
\[\s_k(\la(D^2\lp^1))>1\,\,\mbox{in $\mathbb R^n\setminus\bar E_1.$}\]

When $1\leq m\leq k-1,$ since we have shown $\la(D^2\oab)\in\Gamma_k,$ by Maclaurin's inequality we have
\[\frac{\s_m(\la(D^2\oab))}{C_n^m}\geq\lt(\frac{\s_k(\la(D^2\oab))}{C_n^k}\rt)^{\frac{m}{k}}.\]
Combining with \eqref{ad2.1} we obtain
\[\s_m(\la(D^2\oab))\geq\frac{C^m_n}{(C_n^k)^{\frac{m}{k}}}\lt(1+\frac{2h_k\eta}{k}y\rt)^{\frac{m}{k}}.\]
It follows that
\begin{align*}
\s_m(\la(D^2\lp^1))&\geq y^{\frac{m}{k}}\lt[\frac{C^m_n}{(C_n^k)^{\frac{m}{k}}}\lt(\frac{1}{y}+\frac{2h_k\eta}{k}\rt)^{\frac{m}{k}}
-\sum\limits_{i=1}^m(1+y)^{\frac{m-i}{k}}C_i\lt(\frac{y}{\alpha}\rt)^{(1+\gamma)i-\frac{m}{k}}\frac{1}{\alpha^{\frac{m}{k}}}\rt]\\
&\geq y^{\frac{m}{k}}\lt[\frac{C^m_n}{(C_n^k)^{\frac{m}{k}}}\lt(\frac{1}{y}+\frac{2h_k\eta}{k}\rt)^{\frac{m}{k}}
-\sum\limits_{i=1}^m(1+\alpha)^{\frac{m-i}{k}}C_i\frac{1}{\alpha^{\frac{m}{k}}}\rt].
\end{align*}

Therefore, there exists $\alpha_0=\alpha_0(A, \Lambda, |\tilde\varphi|_{C^2},\beta, k)>0$ such that for any $\alpha>\alpha_0$
\[\s_m(\lambda(D^2\lp^1))>0\,\,\mbox{in $\mathbb R^n\setminus\bar E_1.$}\]
We conclude this subsection with the following proposition.
\begin{proposition}
\label{prosub3}
For $n\geq 3,$ $2\leq k\leq n-1,$ and $A\in\mathcal A_k,$ denote $$\omega_{\alpha, \beta}(x)=\int_1^s(1+\alpha t^{-\beta})^{1/k}dt$$ where $s=\frac{1}{2}x^TAx.$
Then given any $\frac{k}{2}<\beta<\frac{k}{2h_k},$ $\La\geq\beta-1,$ and $\tvarphi\in C^{\infty}(S^{n-1}),$ there exists $\alpha_0=\alpha_0(A, \La, |\tvarphi|_{C^2}, \beta, k)>0$ such that
when $\alpha>\alpha_0$
\be\label{sub9}
\lp^1=\oab+s^{-\La}\tvarphi
\ee
is a smooth strictly $k$-convex subsolution of $\s_k(\lambda(D^2u))=1$ in $\mathbb R^n\setminus E_1$ satisfying
$\lp^1=\tvarphi$ on $\p E_1.$
\end{proposition}

\subsection{Subsolutions of \eqref{1.1} in $\mathbb R^n\setminus D$}
In this subsection, we will follow the idea of \cite{CL03} to glue $\lp$ and $\lp^1$ together to obtain a subsolution of \eqref{1.1} in $\mathbb R^n\setminus D.$

Recall that $\lp=\ba^N-1+\varphi,$ following earlier notations, a direct calculation gives
\[e_a\lp=N\ba^{N-1}\ba_a+\varphi_a,\]
\[\tau_r\lp=N\ba^{N-1}\ba_r,\]
this implies
\[|D\lp|\leq C_0=C_0(N, |\varphi|_{C^1}, \Gamma, A)\,\,\mbox{on $\partial E_1.$}\]
Next, we compute $|D\lp^1|$ on $\p E_1$
\[
\begin{aligned}
e_a\lp^1&=\frac{\partial\lp^1}{\partial s}s_a+s^{-\La}\tvarphi_a\\
&=\lt[(1+\alpha)^{1/k}-\La\tvarphi\rt]s_a+\tvarphi_a,
\end{aligned}
\]
and
\[\tau_r\lp^1=\lt[(1+\alpha)^{1/k}-\La\tvarphi\rt]s_r,\]
where $\tvarphi=\lp|_{\partial E_1}.$
Hence, when $\alpha>\alpha_1=\alpha_1( C_0, \La, |\tvarphi|_{C^1}, A)>0$ we have
\[|D\lp^1|>|D\lp|\,\,\mbox{on $\partial E_1.$}\]
Let $\nu$ be the outward unit normal of $\partial E_1$, i.e., pointing into $\mathbb R^n\setminus\bar E_1,$ since $\lp^1=\lp$
on $\partial E_1,$ we have
\be\label{sub10}
D_\nu\lp^1>D_\nu\lp
\ee
on $\partial E_1.$

We will denote
\be\label{subsolution}
\lu:=\left\{\begin{aligned}
&\lp\,\,&x\in\bar E_1\setminus D,\\
&\lp^1\,\, &x\in\mathbb R^n\setminus E_1.
\end{aligned}\right.
\ee
It is clear that $\lu$ is continuous on $\mathbb R^n\setminus D$. In next section we will show that $\lu$ is a viscosity subsolution of \eqref{1.1}.

\bigskip
\section{Proof of the existence part of theorem \ref{thm1}}
\label{secthm1}
In this section we will prove the existence part of Theorem \ref{thm1}.
We will assume $N, \alpha, \frac{k}{2}<\beta<\frac{k}{2h_k},$ and $\La>\beta-1$ are fixed constants such that the following conditions are satisfied:
\begin{itemize}
\item[(a).] $\lp=\ba^N-1+\varphi$ is strictly $k$-convex and satisfies $\s_k(\la(D^2\lp))>1$ in $E_1\setminus \bar D;$
\item[(b).] $\frac{1}{2}x^TAx+\mu(\alpha, \beta)\geq\varphi$ on $\Gamma;$
\item[(c).] $\lp^1=\oab+s^{-\La}\tvarphi$ satisfies Proposition \ref{prosub3} and equation \eqref{sub10}.
\end{itemize}
Here $\tvarphi=\lp|_{\p E_1}$ is a fixed smooth function determined by $N, \varphi, A$ and $\Gamma.$

Let $\lu$ be the function defined by \eqref{subsolution}, then
in view of of Proposition \ref{prosub2} we get
\[\lu=\lp^1=s+\mu(\alpha, \beta)+O(s^{1-\beta})\,\,\mbox{as $s\goto\infty.$}\]
Now, let $\bar u_R=\frac{1}{2}x^TAx+\mu(\alpha, \beta)+\bar CR^{1-\beta},$ where $\bar C=\bar C(\alpha, \beta)>0$ is chosen such that
$\bar u_R>\lu$ on $\partial E_R.$ Here and in the following, without loss of generality, we will always assume $R\geq R_0\gg1$ is an arbitrarily large constant.
We will also assume $\bar C>0$ is a fixed constant.
Moreover, in view of condition (b) we have $\bar u_R>\lu=\varphi$ on $\Gamma.$

Now, let us consider the following Dirichlet problem
\be\label{dirichlet}
\left\{\begin{aligned}
\s_k(\la(D^2u))&=1\,\,&\mbox{in $\Omega_R:=E_R\setminus\bar D$}\\
u&=\varphi\,\,&\mbox{on $\Gamma:=\partial D$}\\
u&=\bar u_R\,\,&\mbox{on $\partial E_R$}.
\end{aligned}
\right.
\ee
We will show that for any $R\geq R_0\gg1$ there exists a smooth, strictly $k$-convex solution $u_R\in C^{\infty}(E_R\setminus\bar D)$ of \eqref{dirichlet}.

\subsection{$C^0$ estimates}
\begin{lemma}
\label{lem-c0}
Let $u_R$ be the strictly $k$-convex solution of \eqref{dirichlet}, then $u_R$ satisfies
\[\lu<u_R<\bar u_R\,\,\mbox{in $E_R\setminus\bar D$.}\]
\end{lemma}
\begin{proof}
Since $u_R\leq\bar u_R$ on $\partial\Omega_R$ and
$\s_k(\la(D^2u_R))=\s_k(\la(D^2\bar u_R))=1.$
By the strong maximum principle we obtain the second inequality.

Now, we consider the first inequality. We know that for some $\delta<0$ small, we have
\be\label{c01}u_R\geq\lu+\delta,\,\,x\in\bar\Omega_R.\ee
Suppose $\bar\delta$ is the largest number for which inequality \eqref{c01} holds, we will show $\bar\delta=0.$ If not, then $\bar\delta<0$
and there exists $\bar x\in \Omega_R$ such that
\[u_R(\bar x)=\lu(\bar x)+\bar\delta.\]
Since on $\lt(E_R\setminus\bar E_1\rt)\cup\lt(E_1\setminus\bar D\rt)$ we get
\[\s_k(\lambda(D^2\lu))>\s_k(\lambda(D^2u_R)).\]
By the standard maximum principle we know $\bar x\notin\lt(E_R\setminus\bar E_1\rt)\cup\lt(E_1\setminus\bar D\rt).$
We conclude that $\bar x\in\partial E_1.$ However, this is impossible. Since by \eqref{sub10} and \eqref{subsolution} we have,
\[\lim\limits_{x\goto\partial E_1^-}D_\nu\lu<\lim\limits_{x\goto\partial E_1^+}D_\nu\lu\]  while $u_R$ is smooth.
This implies $\bar\delta=0$ and the lemma follows from the strong maximum principle.
\end{proof}

\begin{remark}\label{rmk3.1.1}
With a small modification of the proof of Lemma \ref{lem-c0}, one can show that $\lu$ is a viscosity subsolution of \eqref{1.1}.
\end{remark}

\subsection{$C^1$ estimates}
In order to prove the $C^1$ estimate of $u_R$ on $\Gamma$ we need to construct an upper barrier on $\Gamma$ first.
\subsubsection{Construction of the upper barrier on $\Gamma$}Recall that $\Gamma$ is a smooth $(n-1)$-dimensional manifold, we know that there exists $\eta_0> 0$ such that for any $\hat p\in\Gamma$ and $\eta\leq\eta_0$
there exists $z_{\hat p, \eta}\in D$ satisfying $\bar B_{\eta}(z_{\hat p, \eta})\cap\Gamma=\hat p.$
Now, fix an arbitrary $\hat p\in\Gamma,$ we may choose a new coordinate $\{\hat x_1, \cdots, \hat x_n\}$ of $\mathbb R^n$ such that $\hat p$ is the origin. We also let $\hat x_n$ axis be the unit normal of $\Gamma$ at $\hat p$ pointing into $D.$ We will first restrict ourselves to a small neighborhood of
$\hat p.$ Denote $U_\delta=\{\hat x'=(\hat x_1, \cdots, \hat x_{n-1})\in\mathbb R^{n-1}: |\hat x'|<\delta\}$ for some fixed small constant $\delta>0,$ then near $\hat p$ the boundary $\Gamma$ can be written as a graph over $U_\delta$
\[\gamma(\hat x')=\frac{1}{2}\sum\limits_{\beta=1}^{n-1}\kappa_{\beta}\hat x^2_\beta+O(|\hat x'|^3),\]
where $(\kappa_1, \cdots, \kappa_{n-1})$ are the principle curvature vector of $\Gamma$ at $\hat p.$
Furthermore, the function $\varphi$ can be written as a function over $U_\delta:$
\be\label{varphi-function}
\varphi(\hat x', \gamma(\hat x'))=\varphi(0)+\sum\limits_{\alpha=1}^{n-1}\varphi_{\alpha}(0)\hat x_\alpha+O(|\hat x'|^2).
\ee
 Let
\be\label{c11}\hat u=-C|\hat x-z_{\hat p, \eta}|^{-(n-2)}+C\eta^{-(n-2)}+\varphi(0)+\sum\limits_{\alpha=1}^{n-1}\varphi_\alpha(0)\hat x_\alpha,\ee
where $C>0$ will be determined later.
We note that $\hat u$ is a smooth function defined on $\mathbb R^n\setminus\{z_{\hat p, \eta}\}.$

\begin{lemma}\label{ad-lem1} If $C=C(\varphi, \Gamma)>0$ is chosen to be large enough and $\eta=\eta(\Gamma)\leq\eta_0$ is chosen to be small enough, then $\hat u>\varphi$ on $\Gamma\setminus\{0\}$ and $\hat u(0)=\varphi (0).$\end{lemma}
\begin{proof}
It is clear that $\hat u(0)=\varphi(0),$ we only need to show that for properly chosen constants $C>0$ large and $\eta>0$ small, we have $\hat u>\varphi$ on $\Gamma\setminus\{0\}.$
Note that, by our choice of coordinates, $z_{\hat p, \eta}=(0, \cdots, 0, \eta).$ In this proof, we will always assume $\delta\leq\frac{\eta}{10}$ is a fixed small number. Then on $\Gamma$ when $|\hat x'|<\delta,$
\[
\begin{aligned}
\hat u&=-C(|\hat x'|^2+(\gamma(\hat x')-\eta)^2)^{-\frac{n-2}{2}}+C\eta^{-(n-2)}+\varphi(0)+\sum\limits_{\alpha=1}^{n-1}\varphi_\alpha(0)\hat x_{\alpha}\\
&=C\eta^{-(n-2)}\lt[1-\frac{\eta^{n-2}}{(|\hat x'|^2+\gamma^2-2\eta\gamma+\eta^2)^{\frac{n-2}{2}}}\rt]+\varphi(0)+\sum\limits_{\alpha=1}^{n-1}\varphi_\alpha(0)\hat x_{\alpha}.
\end{aligned}
\]
Since $\gamma(\hat x')=\frac{1}{2}\sum\limits_{\beta=1}^{n-1}\kappa_{\beta}\hat x^2_\beta+O(|\hat x'|^3),$ we can see that when $\eta=\eta(\Gamma)>0$ small, we have
\[|\hat x'|^2+\gamma^2-2\eta\gamma>\la|\hat x'|^2\,\,\mbox{for some $1/2<\la<1.$}\]
This gives
\[\begin{aligned}
\hat u&>C\eta^{-(n-2)}\lt[1-\frac{\eta^{n-2}}{(\la|\hat x'|^2+\eta^2)^{\frac{n-2}{2}}}\rt]+\varphi(0)+\sum\limits_{\alpha=1}^{n-1}\varphi_\alpha(0)\hat x_{\alpha}\\
&=C\eta^{-(n-2)}\lt[1-\frac{1}{\lt(1+\frac{\la|\hat x'|^2}{\eta^2}\rt)^{\frac{n-2}{2}}}\rt]+\varphi(0)+\sum\limits_{\alpha=1}^{n-1}\varphi_\alpha(0)\hat x_{\alpha}\\
&>\frac{C}{(\eta^2+\la|\hat x'|^2)^{\frac{n-2}{2}}}\cdot\frac{n-2}{4}\cdot\frac{\la|\hat x'|^2}{\eta^2}+\varphi(0)+\sum\limits_{\alpha=1}^{n-1}\varphi_\alpha(0)\hat x_{\alpha}.
\end{aligned}
\]
By virtue of \eqref{varphi-function}, it is clear that when $C=C(\varphi, \Gamma)>0$ large and $\eta=\eta(\Gamma)>0$ small, we have $\hat u\geq \varphi$
on $|\hat x'|<\delta.$

Now, we consider the set $\Gamma^c_\delta:=\big\{x\in\mathbb R^n: x\in\Gamma\setminus\{(\hat x', \gamma(\hat x')):|\hat x'|<\delta\}\big\}.$ It is easy to see that when $x\in\Gamma^c_\delta,$ we have $|x-z_{\hat p, \eta}|\geq \eta+\delta_1$ for some small constant $\delta_1=\delta_1(\eta, \Gamma)>0.$ Then
\[
\begin{aligned}
\hat u&\geq C\eta^{-(n-2)}-C(\eta+\delta_1)^{-(n-2)}+\varphi(0)+\sum\limits_{\alpha=1}^{n-1}\varphi_\alpha(0)\hat x_{\alpha}\\
&>\frac{C(n-2)\delta_1}{2\eta^{n-1}}+\varphi(0)+\sum\limits_{\alpha=1}^{n-1}\varphi_\alpha(0)\hat x_{\alpha}.
\end{aligned}
\]
Therefore, when $C=C(\varphi, \Gamma)>0$ large and $\eta=\eta(\Gamma)>0$ small the claim holds.
\end{proof}

\begin{lemma}
\label{lem-c1-b}
Let $u_R$ be the strictly $k$-convex solution of \eqref{dirichlet}, then $u_R$ satisfies
\be\label{c12}|Du_R|<C_\Gamma\,\,\mbox{on $\Gamma,$}\ee
and \be\label{c13} |Du_R|<C_R\,\,\mbox{on $\partial E_R,$}\ee
where $C_\Gamma=C_\Gamma(\Gamma, \varphi)$ is independent of $R$ and $C_R=O(R^{1/2}).$
\end{lemma}
\begin{proof}
By Lemma \ref{lem-c0} we have $\lu<u_R$ in $\Omega_R$ and $u_R=\lu$ on $\Gamma.$
Therefore, we get $D_\nu u_R>D_\nu\lu$ on $\Gamma,$ where $\nu$ is the inward normal of $\Gamma,$ that is, pointing into $\Omega_R.$
Now, for any $\hat p\in\Gamma,$ we consider $\hat u$ that is given in \eqref{c11}. It is easy to verify that
\[\s_1(\la(D^2\hat u))=0<\s_1(\la(D^2 u_R)).\]
Futhermore, we can always choose $C=C(\varphi, \Gamma, A)>0$ large and $\eta=\eta(\Gamma, \varphi, A)>0$ small such that
\[\hat u>\bar u_R>u_R\,\,\mbox{on $\partial E_1.$}\]
In view of Lemma \ref{ad-lem1} and the maximum principle we obtain $\hat u\geq u_R$ in $\overline{E_1\setminus D}.$ Moreover, at $\hat p$ we have
$\hat u=u_R.$ Thus $D_\nu\hat u>D_\nu u_R$ at $\hat p.$ Since $\hat p\in\Gamma$ is arbitrary we conclude \eqref{c12}.

On $E_R,$ let $\lu_R=\la x^TAx-(2\la R-R-\mu(\alpha, \beta)-\bar CR^{1-\beta}),$
where $\la=\la(\mu(\alpha, \beta), R_0, \varphi)>1/2$ such that for any $R>R_0,$ $\lu_R<\varphi$ on $\Gamma.$ Applying the maximum principle we get
\[\lu_R<u_R<\bar u_R\,\,\mbox{in $E_R\setminus\bar D$}.\]
Moreover, on $\partial E_R$ it is clear that $\lu_R=u_R=\bar u_R.$ Therefore, we have
\[D_\nu\lu_R<D_\nu u_R<D_\nu \bar u_R,\]
where $\nu$ is the inward unit normal to $\partial E_R,$ that is, pointing into $\Omega_R.$
This proves \eqref{c13}.
\end{proof}

\begin{lemma}
\label{lem-c1-g}
Let $u_R$ be the strictly $k$-convex solution of \eqref{dirichlet}, then $u_R$ satisfies
\be\label{c14}
\max\limits_{x\in\bar\Omega_R}|Du_R|=\max\limits_{x\in\partial\Omega_R}|Du_R|.
\ee
\end{lemma}
\begin{proof} Differentiating \eqref{dirichlet} with respect to $x_l$ gives
\[\s_k^{ij}u_{ijl}=0.\]
Therefore, we have
\[\s_k^{ij}(|Du_R|^2)_{ij}=2\s_k^{ij}u_lu_{lij}+2\s_k^{ij}u_{li}u_{lj}=2\s_k^{ij}u_{li}u_{lj}\geq 0.\]
By virtue of the maximum principle this lemma is proved.
\end{proof}

\subsection{$C^2$ estimates} In this subsection, we will establish the $C^2$ estimates of $u_R.$ The techniques are used here are the same as \cite{CNS3}. For readers' convenience,
we include the argument here.

\subsubsection{$C^2$ boundary estimates on $\Gamma$} Let $x_0\in\Gamma$ be an arbitrary point. Without loss of generality, we may choose local coordinates
$\{\tx_1, \cdots, \tx_n\}$  in the neighborhood of $x_0$ such that $x_0$ is the origin and $\tx_n$
axis is the inward normal of $\Gamma$ (pointing into $\Omega_R$) at $x_0.$ Then the boundary near $x_0$ can be expressed as
\[\tx_n=\gamma(\tx')=-\frac{1}{2}\sum\limits_{\alpha=1}^{n-1}\kappa_\alpha\tx_\alpha^2+O(|\tx'|^3),\]
where $\kappa_1, \cdots, \kappa_{n-1}$ are the principal curvatures of $\Gamma$ at $x_0$ and $\tx'=(\tx_1, \cdots, \tx_{n-1}).$
Following \cite{CNS3}, we may assume $\varphi$ has been extended smoothly to $\bar\Omega_R$ with $\varphi(0)=0.$ Then we get
\[u_{\ta\tb}(0)=\varphi_{\ta\tb}(0)-u_{\tn}(0)\kappa_\alpha\delta_{\ta\tb}\,\,\mbox{for $\alpha, \beta<n$.}\]
This gives
\be\label{c2.1}
|u_{\ta\tb}(0)|\leq C, \,\,\alpha, \beta<n
\ee
for some $C=C(\varphi, \Gamma, |Du|_{C^0(\Gamma)})>0$ that is independent of $R.$

Next, we estimate $|u_{\ta\tn}(0)|$ for $\alpha<n.$
\begin{lemma}
\label{lem-c2-ib-1}
Let $u_R$ be the strictly $k$-convex solution of \eqref{dirichlet}, then $u_R$ satisfies
\be\label{c2.2}|(u_R)_{\tau\nu}|<C\,\,\mbox{on $\Gamma,$}\ee
where $\tau$ is an arbitrary unit tangent vector of $\Gamma,$ $\nu$ is the inward unit normal of $\Gamma,$ and $C=C(\Gamma, \varphi, |Du|_{C^0(\Gamma)})>0$ is independent of $R.$
\end{lemma}
\begin{proof}
Notice that the boundary $\Gamma$ near $x_0$ can be expressed as
\[\tx_n=\gamma(\tx')=-\frac{1}{2}\sum\limits_{\alpha=1}^{n-1}\kappa_\alpha\tx^2_\alpha+O(|\tx'|^3).\]
Denote $W_R=u_R-\varphi$ and let
\[T:=\partial_{\ta}-\kappa_\alpha(\tx_\alpha\partial_{\tn}-\tx_n\partial_{\ta}),\]
it is clear that on $\Gamma$ near $x_0$ we have
\be\label{c2.3}
TW_R=\lt(\partial_{\ta}+\gamma_{\ta}\partial_{\tn}\rt)W_R+O(|\tx'|^2)=O(|\tx'|^2).
\ee
Set $L:=\frac{\partial\s_k(\la(D^2u))}{\partial u_{ij}}\partial_i\partial_j,$ then we get
\be\label{c2.4}
|L(TW_R)|\leq C_0\sum\s_k^{ii}
\ee
for some $C_0=C_0(\varphi, \Gamma, |Du|_{C^0(\Gamma)})>0.$ In view of Proposition \ref{prosub1}, we know that $\lambda(D^2\lp)\in\Gamma_k$ in $E_1\setminus\bar D.$
Moreover, we also have
\[\s_k(\la(D^2\lp))>1+\delta_0\,\,\mbox{in $B_\e(x_0)\cap\Omega_R,$}\]
where $\e>0$ is a fixed small number and $\delta_0=\delta_0(N)>0$ for some $N$ chosen in Proposition \ref{prosub1}. Therefore, there exists
$\delta_1=\delta_1(\lp, \delta_0)>0$ satisfying
\[\s_k(\la(D^2\lp-\delta_1|\tx|^2))>1\,\,\mbox{in $B_\e(x_0)\cap\Omega_R.$}\]
By the concavity of $\s_k^{1/k}$ we know
\[L(\lp-\delta_1|\tx|^2)>k\,\,\mbox{in $B_\e(x_0)\cap\Omega_R,$}\]
which implies
\[L(u_R-\lp+\delta_1|\tx|^2)<0\,\,\mbox{in $B_\e(x_0)\cap\Omega_R.$}\]
We will denote $h=u_R-\lp+\frac{\delta_1}{2}|\tx|^2,$ then we have $Lh<-\delta_1\sum\limits_i\s_k^{ii}.$
Set $\tilde W=Bh\pm TW_R,$ we may choose $B>0$ large such that $\tilde W\geq0$ on $\partial(B_\e\cap\Omega_R)$ and $L\tilde W<0$
in $B_\e\cap\Omega_R.$ By the maximum principle we conclude $\tilde W>0$ in $B_\e\cap\Omega_R.$ Thus,
at $x_0$ we have $\tilde W_{\tn}>0,$ which gives
\[|(u_R)_{\ta\tn}|<C\]
for some $C=C(\Gamma, \varphi, |Du|_{C^0(\Gamma)})>0$ that is independent of $R.$
\end{proof}

Finally, following the well-known argument of \cite{CNS3} on page 284 and 285 we obtain
\begin{lemma}
\label{lem-c2-ib-2}
Let $u_R$ be the strictly $k$-convex solution of \eqref{dirichlet}, then $u_R$ satisfies
\be\label{c2.5}|(u_R)_{\nu\nu}|<C\,\,\mbox{on $\Gamma,$}\ee
where $\nu$ is the inward unit normal of $\Gamma,$ and $C=C(\Gamma, \varphi, |Du|_{C^0(\Gamma)})>0$ is a constant that is independent of $R.$
\end{lemma}

\subsubsection{$C^2$ boundary estimates on $\partial E_R$} Now we let
$\tilde{u}_R=\frac{1}{R}u_R(\sqrt{R}x),$ then $\tu_R$ satisfies
\be\label{c2.6}
\left\{
\begin{aligned}
\s_k(\la(D^2\tu))&=1\,\,&\mbox{in $\frac{1}{\sqrt{R}}\Omega_R,$}\\
\tu&=\frac{\varphi}{R}\,\,&\mbox{on $\partial\frac{D}{\sqrt R},$}\\
\tu&=1+\frac{\mu(\alpha, \beta)}{R}+\bar CR^{-\beta}\,\,&\mbox{on $\partial E_1.$}
\end{aligned}
\right.
\ee
By Section 5 of \cite{CNS3} we have on $\partial E_1$
\[|D^2\tu_R|\leq C\]
for some constant $C=C(A)>0$ that is independent of $R.$ This in turn implies that
\[|D^2u_R|\leq C\,\,\mbox{on $\partial E_R.$}\] We conclude
\begin{lemma}
\label{lem-c2-ob}
Let $u_R$ be the strictly $k$-convex solution of \eqref{dirichlet}, then $u_R$ satisfies
\be\label{c2.7}|D^2u_R|<C\,\,\mbox{on $\p\Omega_R,$}\ee
for some constant $C=C(\Gamma, \varphi, |Du|_{C^0(\Gamma)}, A)>0$ that is independent of $R.$
\end{lemma}

\subsubsection{$C^2$ global estimates} Finally we prove
\begin{lemma}
\label{lem-c2-gb}
Let $u_R$ be the strictly $k$-convex solution of \eqref{dirichlet}, then $u_R$ satisfies
\be\label{c2.8}|D^2u_R|<C\,\,\mbox{in $\bar\Omega_R,$}\ee
for some constant $C=C(A, \varphi, \Gamma, |Du|_{C^0(\Gamma)})>0$ that is independent of $R.$
\end{lemma}
\begin{proof}
We denote $F(D^2u)=\s_k^{1/k}(\la(D^2u)),$ then \eqref{dirichlet} can be written as
\be\label{c2.9}
F(D^2u)=1\,\,\mbox{in $\Omega_R.$}
\ee
Now differentiating \eqref{c2.9} twice we get
\[F^{ij}u_{ijkk}+F^{pq, rs}u_{pqk}u_{rsk}=0.\]
It is well known that $F$ is a concave function. Therefore, we get
\[F^{ij}u_{ijkk}=F^{ij}(\Delta u)_{ij}\geq 0.\]
By the maximum principle we obtain
\be\label{c2.10}\max\limits_{x\in\bar\Omega_R}\Delta u_R=\max\limits_{x\in\partial\Omega_R}\Delta u.\ee
Since
\[(\Delta u)^2-\sum\limits_{i, j}|u_{ij}|^2=2\s_2(\la(D^2u))>0,\]
\eqref{c2.8} follows from \eqref{c2.10} and Lemma \ref{lem-c2-ob}.
\end{proof}

\subsection{Proof of the existence part of Theorem \ref{thm1}} In this subsection, we will complete the proof of the existence part of Theorem \ref{thm1}.
First, by virtue of equation \eqref{c12} and Lemma \ref{lem-c2-gb} it is easy to see that there exist some constants $C_0, C_1>0$ that are independent of $R$
such that $|Du_R|\leq C_0+C_1\text{dist}(x, \partial D).$ Moreover, applying Evans-Krylov theorem and Schauder estimates, we can obtain higher order estimates of $u_R$
that are independent of $R.$
By the standard maximum principle we know that when $R_1<R_2,$ for any compact set $K\subset\Omega_{R_1}\subset\Omega_{R_2},$ we have $u_{R_2}<u_{R_1}$ on $K.$  Therefore, $\{u_R\}$ is decreasing in $R.$ We conclude
\[u_R(x)\goto u(x)\,\,\mbox{in $C_{loc}^\infty$ topology.}\]
Moreover, $u(x)\in C^\infty(\mathbb R^n\setminus D)$ satisfies
\be\label{c2.11}
\left\{
\begin{aligned}
\s_k(\la(D^2u))&=1\,\,&\mbox{in $\mathbb R^n\setminus\bar D$}\\
u&=\varphi\,\,&\mbox{on $\partial D.$}
\end{aligned}
\right.
\ee
This completes the proof of the existence part of Theorem \ref{thm1}.

\bigskip
\section{Asymptotic behavior near infinity}
\label{secthm2}
In this section, we will prove the second part of Theorem \ref{thm1}, that is, inequality \eqref{1.2} and  Theorem \ref{thm2}.

Recall Lemma \ref{lem-c0} we have
\[\lu<u_R<\bar u_R\,\,\mbox{in $E_R\setminus\bar D.$}\]
By virtue of equations \eqref{sub9}, \eqref{subsolution}, and Proposition \ref{prosub2} we obtain that in $E_R\setminus E_{R_0}$
\[
\begin{aligned}
&s+\mu(\alpha, \beta)+C_1s^{1-\beta}<u_R\\
<&s+\mu(\alpha, \beta)+\bar CR^{1-\beta}\leq s+\mu(\alpha, \beta)+\bar Cs^{1-\beta}.
\end{aligned}
\]
Here $C_1=C_1(\alpha, \beta)<0$ and $R_0\gg1$ is a large constant.
Therefore,  when $|x|>R_0^{1/2}$ is very large we have
\be\label{r1-1}|x|^{2\beta-2}|u-s-\mu(\alpha, \beta)|\leq C\ee
for some $C=C(\alpha, \beta, A)>0.$

Without loss of generality, in the following, we may assume $\mu(\alpha, \beta)=0.$
Now for any $x\in\mathbb R^n\setminus\bar D$ and $|x|=L>2R_0^{1/2},$ let
$u^L(y):=\lt(\frac{4}{L}\rt)^2u\lt(x+\frac{L}{4}y\rt),\,\,|y|\leq2.$
Then  by Lemma \ref{lem-c2-gb} $u^L(y)$ satisfies
\[\s_k(\la(D^2u^L(y)))=1\,\,\mbox{and $|D^2u^L|\leq C$ in $B_2(0).$}\]
Applying Evans-Krylov theorem and Schauder estimates we immediately obtain for any $m\geq 3$
\be\label{r1-2}
|D^m u^L|\leq C\,\,\mbox{in $B_{1}(0)$}
\ee
for some $C=C(m, n, k, |u^L|_{C^2(B_2(0))})>0.$
Now, denote
\[E^L(y):=u^L(y)-8\lt(\frac{x}{L}+\frac{y}{4}\rt)^TA\lt(\frac{x}{L}+\frac{y}{4}\rt)\]
By \eqref{r1-1} we get, on $B_2(0)$
\[|E^L(y)|\leq\lt(\frac{4}{L}\rt)^2\frac{C}{\lt|x+\frac{L}{4}y\rt|^{2(\beta-1)}}\leq\frac{C}{L^{2\beta}}.\]
Moreover, since $D^2E^L(y)=D^2u^L(y)-A,$ we have $F(A+D^2E^L)=F(A)=1.$
Let $F(\xi):=\s_k^{1/k}(\la(\xi_{ij})),$ $F_{\xi_{ij}}:=\frac{\partial F}{\partial\xi_{ij}},$ and
$\hat a_{ij}(y):=\int_0^1F_{\xi_{ij}}(A+\tau D^2E^L(y))d\tau.$
Then we can see that $E^L$ satisfies
\[\hat a_{ij}(y)D_{ij}E^L(y)=0\,\,\mbox{on $B_{1}(0).$}\]
 Note that by Lemma \ref{lem-c2-gb} and the estimate \eqref{r1-2} we have on $B_{1}(0),$
 \[\frac{1}{\La}I<\hat a_{ij}(y)<\La I\,\,\mbox{ and $|D^m\hat a_{ij}(y)|<C$ for $m\geq 1,$}\]
where $\La=\La(n, k, A, |u^L|_{C^2(B_1(0))})>0$ and $C=C(m, n, k, A, |u^L|_{C^2(B_2(0))})>0.$
By Schauder estimates (see \cite{GT83}) we get for any $m\geq 1$
$$\label{r1-3}|D^mE^L(0)|\leq C|E^L(y)|_{C^0(B_1(0))}\leq\frac{C}{L^{2\beta}},$$
where $C=C(m, n, k, A,  |u^L|_{C^2(B_2(0))})>0.$
This yields for any $x\in\mathbb R^n\setminus\bar D$ and $|x|>2R_0^{1/2},$ we have \be\label{r1-3}|D^m E(x)|\leq\frac{C}{|x|^{2\beta-2+m}}\ee for $E(x)=u(x)-\frac{1}{2}x^TAx$ and $C=C(m, n, k, A, |u|_{C^2})>0.$

Now, we can follow the proof of Lemma 3.6 in \cite{CL03} (which is an estimate for linear elliptic equation) and obtain \eqref{1.2}. This completes the proof of Theorem \ref{thm1}.
Repeating the argument above by replacing $\beta$ with $\frac{n}{2}$ gives Theorem \ref{thm2}.

\section*{Acknowledgements}
We are grateful to Jiguang Bao and Cong Wang for their insightful suggestions, which enhanced the exposition of this paper.

\end{document}